\documentclass[12pt,absolute]{mymathart}
\usepackage{mymathmacros}

\newcommand{\extaddress}[1]{\underline{#1}}

\numberwithin{equation}{section}

\renewcommand{\u}{{\tt u}}

\newcommand{\bdyit}[2]
             {{\rule{0pt}{0pt}_{\mbox{$\scriptstyle #2$}}^{\mbox{%
                   $\scriptstyle #1$}} }}

\newcommand{\addu}{\extaddress{\u}}

 \newtheoremstyle{numberedstyle}
   {9pt}
   {9pt}
   {\normalfont}
   {}
   {\bfseries}
   {.}
   {\newline}
   {}

\newtheorem{thm}{Theorem}[section]%
\newtheorem{lem}[thm]{Lemma}%
\newtheorem{prop}[thm]{Proposition}%

\theoremstyle{numberedstyle}
\newtheorem{defn}[thm]{Definition}%

\usepackage{hyperref}

\title{Connected escaping sets of exponential maps}
\author{Lasse Rempe}
\address{Department of Mathematical Sciences, University of Liverpool, Liverpool L69 7ZL, UK}
\email{l.rempe@liverpool.ac.uk}
\thanks{Supported by EPSRC Grant EP/E017886/1 and Fellowship EP/E052851/1.}

\begin{document} 
 \begin{abstract}
  We show that for many parameters $a\in\C$, the set $I(f_{a})$ of points that
   converge to infinity under
   iteration of the exponential map $f_a(z)=e^z+a$ is connected. This includes all parameters
   for which the singular value $a$ escapes to infinity under iteration of $f_a$.
 \end{abstract}
 
 \maketitle

\medskip

 \section{Introduction}
  If $f:\C\to\C$ is a transcendental entire function (that is, a non-polynomial holomorphic
   self-map of the complex plane), the \emph{escaping set} of $f$ is defined as
    \[ I(f) := \{z\in\C: f^n(z)\to\infty\}. \]
   (Here $f^n=\underset{n\text{ times}}{\underbrace{f\circ\dots\circ f}}$ 
   denotes the $n$-th iterate of $f$, as usual.)

  This set has recently received much attention in the study of transcendental dynamics,
   due to the structure it provides to the dynamical plane of such functions. It is neither an
   open nor a closed subset of the complex plane and tends to have interesting
   topological properties. In the simplest cases 
   (see \cite{devaneykrych,baranskihyperbolic,strahlen,boettcher}), the set $I(f)$ is homeomorphic
   to a subset of a ``Cantor Bouquet'' (a certain uncountable disjoint union of curves to 
   $\infty$), and in particular $I(f)$ is disconnected for these functions. It has recently come to light that there are many situations where $I(f)$ is in fact
   connected. Rippon and Stallard showed that this is the case for any entire function
   having a multiply-connected wandering domain \cite{ripponstallardfatoueremenko} and also for
   many entire functions of small order of growth \cite{ripponstallardsmallgrowth}. 
   These examples have
   infinitely many critical values. The latter condition is not necessary, as there
   are even maps with connected escaping set in the family
   \[ f_{a}:\C\to\C;\quad z\mapsto \exp(z)+a\]
   of \emph{exponential maps}, which may be considered 
   the simplest parameter space
   of transcendental entire functions. (These maps have no critical points, and
   exactly one \emph{asymptotic value}, namely the omitted value $a$.)
   Indeed, it was shown in \cite{escapingconnected}
   that the escaping set is connected for the standard exponential map $f_0$, while
   all path-connected components of $I(f)$ are relatively closed and nowhere dense. The proof
   uses previously known results about this particular function, thus leaving open the
   possibility of the connectedness of $I(f_0)$ being a rather unusual phenomenon.

 Motivated by this result, Jarque \cite{jarqueconnected} showed that $I(f_a)$ is connected
  whenever $a$ is a \emph{Misiurewicz parameter}, i.e.\ when
  the singular value $a$ is preperiodic. In this note, we extend his proof
  to a wider class of parameters. Our results suggest that connectedness of the escaping
  set is in fact true for ``most'' parameters for which the singular value belongs to the Julia 
  set $J(f_a)=\cl{I(f_a)}$.\footnote{
   The Julia set is defined as the set of non-normality of the 
       family of iterates of $f$. For certain entire functions, including all exponential maps,
       it coincides with the closure of $I(f_a)$ by 
       \cite{alexescaping,alexmisha}.}
  If $a\notin J(f_a)$, then $f_a$ has an attracting or parabolic
  periodic orbit, and it is well-known that the Julia set, 
  and hence the escaping set, is a disconnected
  subset of $\C$. 

  The main condition used in our paper is the following combinatorial notion, first introduced 
   in \cite{nonlanding}. 

 \begin{defn} \label{defn:accessible}
   We say that the singular value
   $a$ of an exponential map $f=f_a$ 
   is \emph{accessible} if $a\in J(f)$ and  
   there is an injective curve 
   $\gamma:[0,\infty)\to J(f)$ with $\gamma(0)=a$, $\gamma(t)\in I(f)$ for $t>0$ 
   and $\re \gamma(t)\to\infty$ as $t\to\infty$.
 \end{defn}
 \begin{remark}[Remark 1]
  It follows from \cite[Corollary 4.3]{markuslassedierk} that this definition
   is indeed equivalent to the one given in \cite{nonlanding}, and in particular that
   the requirement  
   $\re \gamma(t)\to\infty$ could be omitted. 
 \end{remark}
 \begin{remark}[Remark 2]
  It is not known whether the condition that the singular value $a$ is accessible
   is always satisfied when $a$ belongs to the Julia set  (as far as we know,
   this is an open question even for quadratic polynomials). Known cases include all 
   Misiurewicz parameters, all parameters for which the singular value escapes 
   and a number of others. Compare \cite[Remark 2 after Definition 2.2]{nonlanding}.  
 \end{remark}
 
 If $\gamma$ is as in this definition, then every component of   
  $f^{-1}(\gamma)$ is a curve
  tending to $\infty$ in both directions. 
  The set $\C\setminus f^{-1}(\gamma)$ consists of
  countably many ``strips'' $S_k$ ($k\in\Z$), which we will assume are 
  labelled such that
  $S_k = S_0 + 2\pi i k$ for all $k$. For our purposes, it does not matter 
  which strip is labelled
  as $S_0$, although it is customary to use one of two conventions: either
  $S_0$ is the strip that contains the points $r+\pi i$ for sufficiently large
  $r$, or alternatively the strip containing the singular value $a$ (provided
  that $f(a)\notin \gamma$). 

 For any point $z\in \C\setminus I(f)$, there is a sequence $\addu=\u_0\u_1 \u_2 \dots $
  of integers, called the \emph{itinerary} with respect to
  this 
  partition, such that 
  $f^{j-1}(z)\in S_{\u_j}$ for all $j\geq 0$. Every escaping point 
  whose orbit does not
  intersect the curve $\gamma$ also has such an itinerary. The itinerary
  of the singular value (if it exists) is called the \emph{kneading sequence}
  of $f$. 

  \begin{thm} \label{thm:main}
   Let $f(z)=\exp(z)+a$ be an exponential map. If
   \begin{enumerate}
    \item the singular value $a$ belongs to the escaping set, or
    \item the singular value $a$ belongs to $J(f)\setminus I(f)$  and 
     is accessible with non-periodic kneading
           sequence, 
   \end{enumerate}
   then $I(f)$ is a connected subset of $\C$. 
  \end{thm}
 \begin{remark}[Remark 1]
  All \emph{path-connected} component of $I(f)$ are nowhere dense
   under the hypotheses of the theorem \cite[Lemma 4.2]{nonlanding}. 
 \end{remark}
 \begin{remark}[Remark 2]
  The theorem applies, in particular, to the exponential map $f=\exp$; this gives an
   alternative proof of the main result of \cite{escapingconnected}.
 \end{remark}

Conjecturally, if $f_a$ has a Siegel disk with bounded-type rotation number, then the 
 singular value $a$ is accessible in our sense, and furthermore accessible from the
 Siegel disk. In this case, the kneading sequence would be periodic and the Julia set
 (and hence the escaping set) disconnected. On the other hand, it is plausible that
 the escaping set of $f_a$ is connected whenever $f_a$ does not have a nonrepelling periodic
 orbit. 

 The second half of Theorem \ref{thm:main} does 
  not have a straightforward generalization to
  other families. This is because the proof relies on the fact that 
  $f^{-1}(\gamma)\subset I(f)$ because the singular value $a$ is omitted, and hence connected
  sets of
  nonescaping points cannot cross the partition boundaries.
  In fact, Mihaljevi\'c-Brandt \cite{helenaconjugacy} has 
  shown that for many postcritically preperiodic
  entire functions, including those in the complex cosine family 
  $z\mapsto a\exp(z) + b\exp(-z)$, the escaping set is disconnected.
  On the other hand, the proof of the first part of our theorem should apply to
  much more general functions; 
  in particular, to all cosine maps for which both critical values escape. 

 We recall that \emph{Eremenko's conjecture} \cite{alexescaping}
  states that every connected component of the escaping set of a transcendental entire function
  is unbounded. This is true for all exponential maps \cite{expescaping} and indeed
  for much larger classes of entire functions \cite{strahlen,eremenkoproperty}. 
  Despite progress, the question remains open in general, while it is now known 
  that some related but stronger properties may fail  
  (compare e.g.\ \cite{strahlen}). The connectivity of the escaping set for a wide variety of
  exponential maps illustrates some of the counterintuitive properties one may encounter in the
  study of connected components of a planar set that is neither
  open nor closed (and exposes the difficulties of constructing a counterexample
  should the conjecture turn out to be false). It seems likely that a better 
  understanding of these
  phenomena will provide further insights into Eremenko's conjecture. 

 \subsection*{Structure of the article.} In Section \ref{sec:escapingpoints},
  we collect some background about the escaping set of an exponential map.
  In Section \ref{sec:setsofnonescapingpoints}, we establish an important
  preliminary result. The proof of Theorem \ref{thm:main} is then carried 
  out in Section \ref{sec:proof}, separated into two different cases
  (Theorems \ref{thm:nonperiodic} and \ref{thm:nonendpoint}). 

 \subsection*{Basic notation.} As usual, we denote the complex plane by
  $\C$, and the Riemann sphere by $\Ch=\C\cup\{\infty\}$. The closure
  of a set $A$ in $\C$ and in $\Ch$ will be 
  denoted $\cl{A}$ resp.\ $\hat{A}$. 
  Boundaries will be understood to be 
  taken in $\Ch$, unless explicitly stated otherwise. 

 \section{Escaping points of exponential maps}\label{sec:escapingpoints}

  It was shown by Schleicher and Zimmer \cite{expescaping}
  that the escaping set $I(f_a)$ of any exponential map is
  organized in curves to infinity, called \emph{dynamic rays} or
  \emph{hairs}, which come equipped with a combinatorial structure and
  ordering. We do not require a precise understanding of this 
  structure. Instead, we 
   take an axiomatic approach, collecting here only 
  those properties that will be used in 
  our proofs.

 \begin{prop} \label{prop:properties}
   Let $f(z)=\exp(z)+a$ be an exponential map. 
  \begin{enumerate}
   \item If $a\in I(f)$, then $a$ is accessible in the sense of
     Definition \ref{defn:accessible}.
   \item If $U\subset\C$ is an open set with $U\cap J(f)\neq\emptyset$, then
     there is a curve $\gamma:[0,\infty)\to I(f)$ with $\gamma(0)\in U$ and
     $\re \gamma(t)\to \infty$ as $t\to\infty$. \label{item:toplusinfinity}
  \end{enumerate}
 \end{prop}
 \begin{proof}
  The first statement 
   follows from \cite[Theorem 6.5]{expescaping}. 

  To prove the second claim, we use the fact that there is a collection of
   uncountably many pairwise disjoint curves to $\infty$ in the
   escaping set. (This also follows from \cite{expescaping}, but has
   been known much longer: see \cite{dgh,devaneytangerman}.) 

  Hence there is a curve $\alpha:[0,\infty)\to I(f)$ with
   $\lim_{t\to\infty}|\alpha(t)|=\infty$ and 
   $f^j(a)\notin \alpha$ for all $j\geq 0$. In particular, $f^{-1}(\alpha(0))$ is an infinite set,
   and by Montel's theorem, there exist
   $n\geq 1$ and 
   some $z_0\in U$ such that $f^n(z_0)=\alpha(0)$. We can analytically 
   continue the branch of $f^{-n}$ that takes $\alpha(0)$ to $z_0$ to obtain
   a curve $\gamma:[0,\infty)\to I(f)$ with
   $f^n\circ\gamma = \alpha$ and $\gamma(0)=z_0\in U$. 

 We have $|f(\gamma(t))|\to\infty$ as
   $t\to\infty$. As $|f(z)|\leq \exp(\re(z))+|a|$ for all $z$, we thus have
   $\re\gamma(t)\to+\infty$ as $t\to\infty$, as
   claimed. 
 \end{proof}

\subsection*{Exponentially bounded itineraries} 
 For the rest of this section, fix an exponential map $f=\exp(z)+a$ 
  with accessible
  singular value, and an associated partition into itinerary strips
  $S_j$. Recall that we defined the itinerary of a point only if its orbit
  never belongs to the strip boundaries. 

 It simplifies terminology if we can speak of itineraries for
  \emph{all} points. Hence we adopt the (slightly non-standard)
  convention that any sequence
  $\addu=\u_0 \u_1 \u_2 \dots$ with $f^j(z)\in\cl{S_{\u_j}}$ is called an
  itinerary of $z$. Thus $z$ has a \emph{unique itinerary} if and only if
  its orbit does not enter the strip boundaries. 
 \begin{defn}
  An itinerary $\addu=\u_0 \u_1 \u_2 \dots$ is \emph{exponentially bounded} if there is 
   a number $x\geq 0$ such
   that 
    $2\pi |\u_j| \leq \exp^j(x)$
   for all $j\geq 0$.
 \end{defn}
\begin{remark}[Remark 1]
 At first glance it may seem that the itinerary of every point
  $z\in\C$ is exponentially bounded, 
  since certainly $|f^n(z)|$, and thus $|\im f^n(z)|$, are exponentially bounded
  sequences.
  However,
  in general, we have no a priori control over how the imaginary parts in the strips $S_j$ behave
  as the real parts tend to $-\infty$. 

 Nonetheless, it seems plausible that all points have exponentially bounded 
   itineraries;  
  certainly this is true for well-controlled cases such as Misiurewicz 
  parameters.
  We leave this question aside, as its resolution is not
  required for our purposes.
\end{remark}
\begin{remark}[Remark 2]
 If $z$ does not have a unique itinerary, we take the statement ``$z$ has exponentially
  bounded itinerary'' to mean that all itineraries of $z$ are exponentially bounded.
  However, two itineraries of $z$
  differ by at most $1$ in every entry, so this is equivalent to
  saying that $z$ has at least one exponentially bounded itinerary. 
\end{remark}

\begin{prop} \label{prop:expbounded}
 If $z\in\C$ belongs to the closure of some 
  path-connected component of $I(f)$ (in particular, if $z\in I(f)$ or $z=a$), then 
  $z$ has exponentially bounded itinerary.
\end{prop}
\begin{proof}
 Let $z_0\in I(f)$, and let $\addu$ be an itinerary of $z_0$. 
  Then $\re f^j(z_0)\to+\infty$, and in particular there exists $R\in\R$ such that
  $\re f^n(z_0)\geq R$ for all $j\geq 0$. 

  The domain $S_0$ is bounded by two components of $f^{-1}(\gamma)$. Each of these
   has bounded imaginary parts 
   in the direction where the real parts tend to $+\infty$.
   (In fact, each preimage component is asymptotic to a straight line 
    $\{\im z = 2\pi k\}$ for some $k\in\Z$, but we do not require this fact.)
   In particular, 
    \[ M := \sup\{ |\im z|:\ z\in \cl{S_0},\  \re z \geq R\} <\infty.\]

  Then it follows that
   $|\im f^j(z_0)-2\pi \u_j| \leq M$ for all $j$, and hence
    \[ 2\pi |\u_j| \leq |\im f^j(z_0)|+M. \]
  Set
   $\alpha := \ln(3(|a|+M+2))$. Elementary calculations give 
   \begin{align} \notag
    \exp(|z|+\alpha)&= 3(|a|+M+2)\exp(|z|) \geq
      \exp(|z|) + 2(|a|+M+2) \\ \label{eqn:elementary}&\geq
      \exp(\re z) + |a| + M + \ln 3 + (|a|+M+2) \\ \notag&\geq
      \exp(\re z) + |a|+ M + \alpha \geq |f(z)|+M+\alpha
   \end{align}
   for all $z\in\C$. It follows that
    \[ 2\pi |\u_j| \leq |\im f^j(z_0)|+M \leq |f^j(z_0)|+M \leq \exp^j(|z_0|+\alpha) \]
   for all $j\geq 0$, so $z_0$ has exponentially bounded itinerary. 

 Also, it is shown in \cite{markuslassedierk} 
  that the partition boundaries, i.e.\ the components of 
  $f^{-1}(\gamma)$, are path-connected components of $I(f)$ (where $\gamma\subset I(f)$ 
  is the curve connecting the singular value to infinity). So if $C$ is the path-connected
  component of $I(f)$ containing $z_0$, then 
  $f^j(C)\subset\cl{S_{\u_j}}$, and hence $f^j(\cl{C})\subset\cl{S_{\u_j}}$, for all 
  $j\geq 0$. So all points in $\cl{C}$ have exponentially bounded itinerary, as claimed.
\end{proof}

\subsection*{Escaping endpoints}
There are two types of escaping points: 

 \begin{defn} \label{defn:endpoint}
  Suppose that $f_a$ is an exponential map and $z\in I(f)$. 
   We say that $z$ is a \emph{non-endpoint} if there is an 
   injective curve $\gamma:[-1,1]\to I(f_a)$ with
   $\gamma(0)=z$; otherwise $z$ is called an \emph{endpoint}.
 \end{defn}

 It follows from \cite{markuslassedierk} that this coincides with
  the classification into ``escaping endpoints of rays'' and
  ``points on rays'' given in \cite{expescaping}; we use the above
  definition here because it is easier to state. In \cite{expescaping},
  escaping endpoints were completely classified; we only require the
  following fact. 

 \begin{prop}[{\cite{expescaping}}]
  Let $f_a$ be an exponential map with $a\in I(f)$ (so in particular $a$ is accessible), and suppose
   $a$ is an endpoint. 
   Then the kneading sequence of $f$ is unique and unbounded.
 \end{prop}

 For exponential maps with an attracting fixed point, any non-endpoint is
  inaccessible from the
  attracting basin \cite{devgoldberg}. The following is a variant of this fact that holds
  for every exponential map. 

 \begin{prop} \label{prop:comb}
  Suppose that $f=\exp(z)+a$ is an exponential map and suppose that
   $z\in I(f)$ is not an endpoint. Then any closed connected
   set $A\subset\C$ with $z\in A$ and $\#A>1$
   contains uncountably many escaping points. 
 \end{prop}
 \begin{proof}[Sketch of proof]
  The idea is that any
   path-connected component of $I(f)$ is accumulated on
   both from above and below by other such components. This is by
   now a well-known argument; see e.g.\
   \cite[Lemma 3.3]{nonlanding} and \cite[Lemma 13]{expbifurcationlocus},
   where it is used in a slightly different context. We
   provide a few more details for completeness.
  
  We may assume that $A$ intersects only countably many
   different path-connected
   components of $I(f)$; otherwise we are done. Let
    $\gamma:[-1,1]\to I(f)$ be as in Definition \ref{defn:endpoint},
    with $\gamma(0)=z$. 
  
  Then there are two sequences $\gamma^+_n:[-1,1]\to I(f)$ and
   $\gamma^-_n:[-1,1]\to I(f)$ of curves that do not intersect $A$ and
   that converge locally 
   uniformly to $\gamma$ from both sides of $\gamma$. Since $A$ is closed,
   it follows that we must have either $\gamma\bigl([-1,0]\bigr)\subset A$
   or $\gamma\bigl([0,1]\bigr)\subset A$.
 \end{proof}

 \section{Closed subsets of non-escaping points}
   \label{sec:setsofnonescapingpoints}
  Let us say that a set $A\subset\C$ \emph{disconnects} the set 
   $C\subset\C$ if
   $C\cap A=\emptyset$ and 
  (at least) two different connected components of $\C\setminus A$ intersect $C$.
  The following lemma was used in \cite{jarqueconnected} to prove the connectivity
  of the escaping set for Misiurewicz exponential maps. 
  \begin{lem}\label{lem:disconnected}
   Let $C\subset\C$. Then $C$ is disconnected if and only if there is a closed connected set
    $A\subset\C$ that disconnects $C$. 
  \end{lem}
  \begin{proof} 
   The ``if'' part is trivial. If $C$ is disconnected, then by the 
    definition of connectivity
    there are two points $z,w\in C$ and an open set $U\subset\C$ with 
    $\partial U\cap C=\emptyset$
    such that $z\in U$ and $w\notin U$. By passing to a connected component
    if necessary, we may assume that $U$ is connected. Let
    $V$ be the connected component of 
    $\Ch\setminus \hat{U}$ that contains $w$. Then $V$ is simply connected
    with $\partial V\subset \partial U\subset\Ch\setminus A$. It follows that
    $\C\setminus V$ has exactly one connected component. Thus
    $A := \partial V\cap \C$ is a closed connected set that
    disconnects $z$ and $w$, as required.
  \end{proof}

 Thus, in order to prove the connectedness of the escaping set,
  we need to study closed connected sets of non-escaping points and show
  that these cannot disconnect $I(f)$. The following proposition
  will be the main ingredient in this argument.

 \begin{prop} \label{prop:boundedimage}
  Let $f=f_a$ be an exponential map with accessible singular value $a$.
   Let $A\subset\C$ be closed and connected. Suppose that furthermore the points in 
   $A$ have uniformly exponentially bounded itineraries, i.e.\ there exists
   a number $x$ with the following property: if $n\geq 0$ and $\u\in\Z$ such that
   $f^n(A)\cap \cl{S_{\u}}\neq\emptyset$, then $2\pi|\u|\leq \exp^n(x)$. 

  If $A\cap I(f)$ is bounded, 
   then there is $n\geq 0$ such that $f^n(A)$ is bounded.
 \end{prop}

This
  is essentially a (simpler) variant
 of \cite[Lemma 6.5]{nonlanding}, and
 can be proved easily in the same manner using
 the combinatorial terminology of that paper.
 Instead, we give an alternative proof
  --- quite similar to the proof of the main theorem of
  \cite{jarqueconnected} --- that does not require
 familiarity with these concepts.

 \begin{proof}
  We prove the converse, so suppose that $f^n(A)$ is unbounded for all $n$. 
   (In particular, $A$ is nonempty.) 
   We need to show that $A\cap I(f)$ is unbounded.

  Similarly as in the proof of Proposition \ref{prop:expbounded}, set 
    \[ M := \sup\{ |\im z|:\ z\in S_0,\  \re z \geq 0\}.\]
   and $\alpha := \ln(3(|a|+M+2))$. Also pick some $z_0\in A$ and let
   $x_0 \geq \max(|z_0|,x)+\alpha$ be arbitrary.
   The hypotheses and (\ref{eqn:elementary}) imply that
     \begin{equation}\label{eqn:imparts}
       |\im f^n(z)| \leq \exp^n(x) + M \leq  \exp^n(x_0) \end{equation}
    whenever $z\in A$ and $n\geq 0$ such that $\re f^n(z)\geq 0$. Also, 
    again by (\ref{eqn:elementary}), 
      \begin{equation}\label{eqn:z0} |f^n(z_0)|\leq \exp^n(x_0). \end{equation}
 
  Let $n\in\N$. Recall that  $f^n(A)A$ is connected and unbounded by assumption.
   Hence by (\ref{eqn:z0}), there exists some
   $z_n\in A$ with
     \[ |f^n(z_n)| = \exp^n(x_0). \]
  We claim that 
   \begin{equation} \label{eqn:pullback}
     \exp^j(x_0)-1\leq |f^j(z_n)| \leq 2\exp^j(x_0)+1
   \end{equation}
  for $j=0,\dots,n$. Indeed, if $j<n$ is such that
  (\ref{eqn:pullback}) is true for $j+1$, then
    \begin{align*}
      \re f^j(z_n) &= \ln|f^{j+1}(z_n)-a| \geq
            \ln(|f^{j+1}(z_n)|-|a|) \\ 
     &\geq \ln(\exp^{j+1}(x_0)-|a|-1) = 
             \exp^j(x_0) - \ln\frac{\exp^{j+1}(x_0)}{\exp^{j+1}(x_0)-|a|-1} \\
      &\geq \exp^j(x_0) - \ln 2 > \exp^j(x_0) - 1. \end{align*}
  Similarly, we see that
   \[
     \re f^j(z_n) \leq \exp^j(x_0)+1. \]
   Together with \eqref{eqn:imparts}, this yields~\eqref{eqn:pullback} for $j$. 
    
  Now let $z$ be any accumulation point of the sequence $z_n$; since
   $a$ is closed (and the sequence is bounded), we have $z\in A$.
   By continuity, \eqref{eqn:pullback} holds also for $z$, and hence
   $z\in A\cap I(f)$. As $x_0$ can be chosen arbitrarily large, we have shown that
   $A\cap I(f)$ is unbounded, as required. 
 \end{proof}

\section{Proof of Theorem \ref{thm:main}} \label{sec:proof}

 The following two  lemmas  study the properties of sets that can
  disconnect the escaping set of an exponential map with accessible singular value.

 \begin{lem} \label{lem:Aunbounded}
  Let $f$ be an exponential map and
   suppose that $A\subset\C\setminus I(f)$ 
   disconnects the escaping set. Then the real parts of $A$ are not bounded
   from above.
 \end{lem}
 \begin{proof}
  This follows immediately from 
   Proposition~\ref{prop:properties}~(\ref{item:toplusinfinity}).
 \end{proof}

 \begin{lem} \label{lem:stillseparating}
  Let $f=f_a$ be an exponential map with accessible singular
  value $a$. 

  Suppose that $A\subset\C\setminus I(f)$ is connected and
   disconnects the escaping set. Then
  \begin{enumerate}
   \item If the real parts of $A$ are bounded from 
    below, then $f(A)$ also disconnects the escaping set.
   \item The common itinerary of the points in $A$ is exponentially bounded.
  \end{enumerate}
 \end{lem}
 \begin{proof}
  Let $\gamma$ be the curve from Definition \ref{defn:accessible}.
   Let $U$ be the component of $\C\setminus A$ that contains a left half plane, and let
   $V\neq U$ be another component of $\C\setminus A$ with $V\cap I(f)\neq\emptyset$.
   (Such a component exists by assumption.) Every component of
   $f^{-1}(\gamma)$ intersects every left half plane. Thus $f^{-1}(\gamma)\subset U$, and
   in particular $\cl{V}\cap f^{-1}(\gamma)=\emptyset$. 

  This means that $\cl{V}$ is contained in a single itinerary domain $S_j$
   and the real parts in $V$ are bounded from below. 
   As $f|_S$ is 
   a conformal isomorphism between $S_j$ and $\C\setminus \gamma$, it follows
   that $f(V)$ is a component of $\C\setminus f(A)$ that intersects
   the escaping set but does not intersect $\gamma$. Hence $f(A)$ disconnects
   the escaping set, which proves the first claim.

  Note that the second claim is trivial if $\cl{A}$ intersects the escaping set,   since every
   escaping point has exponentially bounded itinerary by Proposition
   \ref{prop:expbounded}. So we may suppose that
   $A$ is a closed set. Also recall that $a$ has exponentially bounded itinerary, which means that
   we may assume that 
   $a\notin \cl{f^n(A)}$ for all $n\geq 0$. Then the real parts of $f^n(A)$ are
   bounded from below for all $n$. 

  Hence for all $n\geq 0$, 
   $f^n(A)$ is a closed subset of $\C$ with real parts bounded from below and disconnecting
   $I(f)$. Let us say that $f^n(A)$ \emph{surrounds} a set $X\subset\C\setminus f^n(A)$ if
   $X$ does not belong to the component of $\C\setminus f^n(A)$ that contains a left half plane. 
    
  By assumption, $A$ surrounds some escaping point $z_0$. We claim that, for every $n\geq 0$,
   \begin{enumerate}
    \item[(*)] $f^n(A)$ surrounds either $f^n(z_0)$ or $f^j(a)$ for some $j< n$.
   \end{enumerate}
  This follows by induction using the same argument similarly
  as in the first part of
  the proof. 
   Indeed, 
  let $w=f^n(z_0)$ or $w=f^j(a)$ be the point surrounded by $f^n(A)$ by the 
  induction
   hypothesis, let $U$ be the component of $\C\setminus f^n(A)$ containing $w$, and let
   $S_j$ be the itinerary strip containing $f^n(A)$ and hence $U$. Now
   $f:S_j\to\C\setminus\gamma$ is a conformal isomorphism. Thus either $f(U)$, and hence
   $f(w)$, is surrounded
   by $f^{n+1}(A)$, or $U$ is mapped to the component of $\C\setminus f^{n+1}(A)$ that contains
   a left half plane and $f^{n+1}(A)$ surrounds $\gamma$, 
   and hence $a$. The induction is complete in either case.

  Because 
   $z_0$ and $a$ both have exponentially bounded itineraries, it follows from (*) that all points in
   $A$ do also. 
 \end{proof}

 Now we are ready to prove Theorem \ref{thm:main}. We begin by treating the 
  case where $f$ has a unique and non-periodic kneading sequence. This includes
  the second case of Theorem \ref{thm:main}, as well as the case of all
  escaping endpoints.

 \begin{thm} \label{thm:nonperiodic}
  Suppose that $f=f_a$ is an exponential map with accessible singular value with 
   unique kneading sequence $\addu=\u_0 \u_1 \u_2 \dots$. 
   If $\addu$ is not periodic, then $I(f)$ is connected.
 \end{thm}
 \begin{proof} 
  We prove the converse. So suppose that $I(f)$ is disconnected; we must show
   that $\addu$ is periodic. By Lemma \ref{lem:disconnected}, there exists
   a closed connected set $A\subset\C$ that disconnects the set of escaping points. 
   Then all 
   points of $A$ have a common itinerary $\addu'$, and this itinerary is exponentially
   bounded by Lemma \ref{lem:stillseparating}. Note that
 \begin{enumerate}
  \item[(*)] \emph{If $k\geq 0$ is such that $f^k(A)$ is unbounded to the left,
   then $a\in\cl{f^{k+1}(A)}$, and hence $\sigma^{k+1}(\addu')=\addu$.}
 \end{enumerate}
  (Here $\sigma$ denotes the shift map; i.e.\
    $\sigma(\u_0 \u_1 \u_2 \dots) = \u_1 \u_2 \dots$.) 

 By Proposition \ref{prop:boundedimage}, $f^k(A)$ is bounded for some $k$; 
   let $k_1$ be minimal with this property.
   Since $A$ is unbounded by Lemma \ref{lem:Aunbounded}, we must have
   $k_1>0$, and since $f^{k_1-1}(A)$ is contained in one of the domains
   $S_j$, it follows that $f^{k_1-1}(A)$ is unbounded to the left. 

 Now let $k_0$ be the minimal number for which $f^{k_0}(A)$ is unbounded
  to the left. By Lemma \ref{lem:stillseparating}, 
  $f^{k_0}(A)$ also disconnects the escaping set, and hence is unbounded
  to the right by Lemma \ref{lem:Aunbounded}. Thus
  $f^{k_0+1}(A)$ is unbounded, and therefore $k_0+1<k_1$ by definition.

 So (*) implies that
   \[ \sigma^{k_1}(\addu')= \addu = \sigma^{k_0+1}(\addu'), \]
 and hence $\sigma^{k_1-k_0-1}(\addu)=\addu$. Thus we have seen that
 $\addu$ is periodic, as claimed.
 \end{proof}

 We now complete the proof of Theorem \ref{thm:main} by covering the case
  where the singular value is escaping but not an endpoint. (Note that there
  are parameters that satisfy the hypotheses of both Theorem
  \ref{thm:nonperiodic} and Theorem \ref{thm:nonendpoint}.)

 \begin{thm} \label{thm:nonendpoint}
  Suppose that $f(z)=\exp(z)+a$ is an exponential map with $a\in I(f)$ such that $a$ is 
   a non-endpoint. Then $I(f)$ is connected.
 \end{thm}
 \begin{proof}
  The singular value is accessible by Proposition \ref{prop:properties}. As shown in the proof of
   Theorem \ref{thm:nonperiodic}, if $I(f)$ was disconnected, 
   there would be an unbounded, closed, connected set
   $A\subset \C\setminus I(f)$ and some number $k_0$ such that
   $f^{k_0+1}(A)\cup\{a\}$ is closed and connected. But this is impossible
   by the Proposition \ref{prop:comb}.  
 \end{proof}

\bibliographystyle{hamsplain}
\bibliography{/Latex/Biblio/biblio}

\end{document}